\documentclass[a4paper,12pt,reqno]{amsart}

\usepackage{a4wide, amssymb, amsmath, amsthm, graphics, comment, xspace, enumerate}
\usepackage{graphicx}
\usepackage{algorithmic, algorithm}

\input amssym.def
\input amssym

\allowdisplaybreaks

\newcommand{\NN}{\mathbb N}

\newcommand{\CC}{\mathbb C}
\newcommand{\RR}{\mathbb R}
\newcommand{\ZZ}{\mathbb Z}
\newcommand{\ds}{\displaystyle}

\newcommand{\DD}{\mathcal D}
\newcommand{\SSS}{\mathcal S}

\newtheorem{theorem}{Theorem}[section]
\newtheorem{proposition}[theorem]{Proposition}
\newtheorem{lemma}[theorem]{Lemma}
\newtheorem{corollary}[theorem]{Corollary}

\theoremstyle{remark}
\newtheorem{remark}[theorem]{Remark}

\theoremstyle{definition}
\newtheorem{definition}[theorem]{Definition}

\allowdisplaybreaks
\numberwithin{equation}{section}

\newcommand{\beq}{\begin{eqnarray}}
\newcommand{\eeq}{\end{eqnarray}}

\newcommand{\beqs}{\begin{eqnarray*}}
\newcommand{\eeqs}{\end{eqnarray*}}
\begin{document}

\title{Convolution with the kernel $e^{s\langle x\rangle^q}, q\geq 1, s>0$ within ultradistribution spaces}

\author[S. Pilipovi\'c]{Stevan Pilipovi\'c}
\address{Stevan Pilipovi\'c, Department of Mathematics and Informatics,
University of Novi Sad, Trg Dositeja Obradovi\'{c}a 4, 21000 Novi Sad, Serbia}
\email{stevan.pilipovic@dmi.uns.ac.rs}

\author[B. Prangoski]{Bojan Prangoski}
\thanks{B. Prangoski was partially supported by the bilateral project ``Microlocal analysis and applications'' funded by the Macedonian and Serbian academies of sciences and arts}
\address{Bojan Prangoski, Faculty of Mechanical Engineering\\ Ss. Cyril and Methodius University in Skopje \\ Karpos II bb \\ 1000 Skopje \\ Macedonia}
\email{bprangoski@yahoo.com}

\author[\DJ. Vuckovi\' c]{\DJ or\dj e Vu\v{c}kovi\' c}
\address{\DJ or\dj e Vu\v ckovi\'c, Mathematical Institute of the Serbian Academy of Sciences and Arts, Knez Mihailova 36, 11000 Belgrade, Serbia}
\email{djordjeplusja@gmail.com}

\thanks{The work presented in this paper  is partially supported by Ministry of Education and  Science, Republic of Serbia, project no. 174024.}



\begin{abstract}
We consider the existence of convolution of Roumieu type ultradistribution with the kernel $e^{s(1+|x|^2)^{q/2}}, q\geq 1$, $s\in\RR\backslash\{0\}$.
\end{abstract}
\maketitle

\section{Introduction}
In  time frequency analysis, when one needs to analyse modulation or Wiener amalgam spaces with  super-exponential weights,  or more generally, to study
 ultradistributions of quasi-analytical classes, kernels of the form $e^{s(1+|x|^2)^{q/2}}, q\geq 1$, $s\in\RR\backslash\{0\}$ are involved.  Our aim is to show that it is possible to study convolution with such kernel beyond a very general theory we have developed in \cite{BSV}.

The problem of when a distribution in $\DD'(\RR^d)$ admits a $\DD'$-convolution with the Gaussian $e^{s|x|^2}$, $s\in\RR\backslash\{0\}$, was studied in \cite{Wagner} and the distributions that do satisfy this property were characterised by their growth behaviour. Later, the same problem was also solved in the setting of non-quasianalytic ultradistributions in \cite{pp} and applied in extending the Anti-Wick quantisation to more general symbol classes. Our aim in this article is to find conditions on the Gelfand-Shilov ultradistributions in $\SSS'^{\{M_p\}}_{\{p!^{1/q_1}\}}(\RR^d)$ which are sufficient for the existence of convolution with the function $x\mapsto e^{s\langle x\rangle^q}$, $s\in\RR\backslash\{0\}$, $q_1>q\geq 1$, where, as standard, $\langle x\rangle=(1+|x|^2)^{1/2}$ (of course, when $q=2$, $e^{s\langle x\rangle^2}$ is just a constant multiple of the Gaussian). Here $\{M_p\}_{p\in\NN}$ is a Gevrey type sequence, bounded from below by $\{p!\}_{p\in\NN}$, which controls the growth of the derivatives of the test functions.\\
\indent The definition of the convolution of $S\in \SSS'^{\{M_p\}}_{\{p!^{1/q_1}\}}(\RR^d)$ with $e^{s\langle \cdot\rangle^q}$ we employ is the following: we say $S$ is convolvable with $e^{s\langle \cdot\rangle^q}$ if for every test function $\varphi\in \SSS^{\{M_p\}}_{\{p!^{1/q_1}\}}(\RR^d)$, $(\varphi*e^{s\langle \cdot\rangle^q})S\in \DD'^{\{M_p\}}_{L^1}(\RR^d)$ (the ultradistributional analogue of the Schwartz space $\DD'_{L^1}(\RR^d)$) and $S*e^{s\langle \cdot\rangle^q}\in \SSS'^{\{M_p\}}_{\{p!^{1/q_1}\}}(\RR^d)$ is defined as
$$
\langle S*e^{s\langle \cdot\rangle^q},\varphi\rangle=\langle (\varphi*e^{s\langle \cdot\rangle^q})S,1\rangle
$$
(see Subsection \ref{structure and topology} for the reasons why this is indeed a well defined element of $\SSS'^{\{M_p\}}_{\{p!^{1/q_1}\}}(\RR^d)$). When $q_1\leq 1$, this reduces to the definition of $\SSS'^{\{M_p\}}_{\{p!^{1/q_1}\}}$-convolution of ultradistributions, see \cite[Definition 5.7 and Theorem 5.8]{BSV} (see also \cite{DPPV,PB,PK,AJs}) and is analogous to one of the equivalent formulations of $\SSS'$-convolution in the distributional setting (see \cite{SchwartzV,shiraishi}; see also \cite{BOO,div,orttt,ortwagn}).\\
\indent When studying the convolution with the Gaussian, one of the key ingredients employed in \cite{Wagner} and \cite{pp} is the fact that $e^{s|x-t|^2}$ splits into a product of two parts: $e^{s|x|^2}$ and $e^{-2sxt+s|t|^2}$. However, when $q\neq 2$ this is no longer the case, and thus we can not apply the same ideas from \cite{Wagner} and \cite{pp}. In fact, a significant part of Section \ref{suff-cond-ee} is devoted to finding good bounds on the growth of the derivatives of $e^{s\langle x\rangle^q}$ before we provide  sufficient conditions in Theorem \ref{cond-for-exi-con-tt}. It is important to emphasise that, when $q=2$, these are analogous to the ones given in the distributional setting as well as in the non-quasianalytic ultradistributional setting (see the remark at the very end of Section \ref{suff-cond-ee}).\\
\indent In Section \ref{neces-con-sdd} we give also necessary conditions for the existence of the convolution with $e^{s\langle \cdot\rangle^q}$, when $s>0$. When $q=1$ they turn out to be the same as the sufficient conditions, and thus providing  a characterisation of the ultradistributions in $\SSS'^{\{M_p\}}_{\{p!^{1/q_1}\}}(\RR^d)$ which are convolvable with $e^{s\langle \cdot\rangle}$, $s>0$.\\
\\
\indent The main results of the paper can be summarised as follows (see Theorem \ref{cond-for-exi-con-tt} and Corollaries \ref{exis-con-cor-ls}, \ref{exi-con-whe-qq} and \ref{cort-l-stv}):\\
{\it
Let $q_1>q\geq 1$, $s\in\RR\backslash\{0\}$ and $S\in\SSS'^{\{p!\}}_{\{p!^{1/{q_1}}\}}(\RR^d)$.\\
\indent 1. If
$$
e^{s\langle \cdot\rangle^q}e^{k\langle \cdot\rangle^{(q-1)q_1/(q_1-1)}} S\in \DD'^{\{p!\}}_{L^1}(\RR^d),\quad \mbox{for all}\,\, k\geq 0,
$$
then the $\SSS'^{\{p!\}}_{\{p!^{1/q_1}\}}$-convolution of $S$ and $e^{s\langle \cdot\rangle^q}$ exists.\\
\indent 2. When $q=1$, the $\SSS'^{\{p!\}}_{\{p!^{1/q_1}\}}$-convolution of $S$ and $e^{s\langle \cdot\rangle}$, $s>0$, exists if and only if $e^{s\langle \cdot\rangle}S\in \DD'^{\{p!\}}_{L^1}(\RR^d)$.\\
\indent 3. Let
$S\in\SSS'^{\{p!^{2-1/q}\}}_{\{p!^{1/q_1}\}}(\RR^d)$. If the
$\SSS'^{\{p!^{2-1/q}\}}_{\{p!^{1/q_1}\}}$-convolution of $S$ and $e^{s\langle \cdot\rangle^q}$, $s>0$, exists then the $\SSS'^{\{p!^{2-1/q}\}}_{\{p!^{1/q_1}\}}$-convolution of $S$ and $e^{s'\langle \cdot\rangle^q}$ also exists for all $s'<s$, $s'\neq 0$.
}\\
\\
\indent Note that the Beurling ultradistribution spaces give even more flexibility in the analysis of these problems, so most of the results (with appropriate changes) can be easily transferred to the ultradistribution spaces of Beurling class.

\section{Preliminaries}

Let $\{M_p\}_{p\in\mathbb{N}}$ be a sequence of positive numbers satisfying $M_0=M_1=1$. Unless stated otherwise, we will always assume the following conditions on $\{M_p\}_{p\in\NN}$ (cf. \cite{Komatsu1}):\\
\indent $(M.1)$ $M_{p}^{2} \leq M_{p-1} M_{p+1}$, $p \in\ZZ_+$;\\
\indent $(M.2)$ $\ds M_{p} \leq c_0H^{p} \min_{0\leq q\leq p} \{M_{p-q} M_{q}\}$, $p,q\in \NN$, for some $c_0,H\geq1$;\\
\indent $(M.5)$ there exists $q>0$ such that $M_p^q$ is strongly
non-quasianalytic (see \cite{Komatsu1}), i.e., there exists $c_0\geq 1$ such that $
\sum_{j=p+1}^{\infty}{M_{j-1}^q}/{M_j^q}\leq c_0
p{M_p^q}/{M_{p+1}^q}$, $\forall p\in\ZZ_+$;\\
\indent $(M.6)$ $p!\subset M_p$, i.e. there exist $C,L>0$ such that $p!\leq CL^pM_p$, $\forall p\in\NN$.\\
For each $\alpha\in\NN^d$, we set $M_{\alpha}=M_{|\alpha|}$. Note that the Gevrey sequences $\{p!^{\kappa}\}_{p\in\NN}$, with $\kappa\geq 1$, satisfy all of the above conditions.\\
\indent Let $q>0$. For every $h>0$, we denote by $\SSS^{M_p,h}_{p!^{1/q},h}(\RR^d)$ the Banach space of all $\varphi\in\mathcal{C}^{\infty}(\RR^d)$ for which the following norm is finite
\beqs
\sup_{\alpha\in\NN^d}h^{|\alpha|}\|e^{h|\cdot|^q} \partial^{\alpha}\varphi\|_{L^{\infty}(\RR^d)}/M_{\alpha}<\infty.
\eeqs
Recall, the Gelfand-Shilov test space $\SSS^{\{M_p\}}_{\{p!^{1/q}\}}(\RR^d)$ is defined as
\beqs
\SSS^{\{M_p\}}_{\{p!^{1/q}\}}(\RR^d)=\lim_{\substack{\longrightarrow \\ h\rightarrow 0^+}} \SSS^{M_p,h}_{p!^{1/q},h}(\RR^d).
\eeqs
Its strong dual $\SSS'^{\{M_p\}}_{\{p!^{1/q}\}}(\RR^d)$ is the space of Gelfand-Shilov ultradistributions; we refer to \cite{PilipovicK,GS,BSV} for its properties. Next, for each $h>0$, let $\DD^{M_p,h}_{L^{\infty}}(\RR^d)$ be the Banach space of all $\varphi\in\mathcal{C}^{\infty}(\RR^d)$ such that $\sup_{\alpha\in\NN^d}h^{|\alpha|} \|\partial^{\alpha}\varphi\|_{L^{\infty}(\RR^d)}/M_{\alpha}<\infty$, and define
\beqs
\DD^{\{M_p\}}_{L^{\infty}}(\RR^d)=\lim_{\substack{\longrightarrow\\ h\rightarrow 0^+}} \DD^{M_p,h}_{L^{\infty}}(\RR^d);
\eeqs
it is a complete barrelled $(DF)$-space (see \cite[Section 4.3]{dpv-2}). Fix $q>0$ and denote by $\dot{\mathcal{B}}^{\{M_p\}}(\RR^d)$ the closure of $\SSS^{\{M_p\}}_{\{p!^{1/q}\}}(\RR^d)$ in $\DD^{\{M_p\}}_{L^{\infty}}(\RR^d)$ and let $\DD'^{\{M_p\}}_{L^1}(\RR^d)$ be the strong dual of $\dot{\mathcal{B}}^{\{M_p\}}(\RR^d)$. As $\SSS^{\{M_p\}}_{\{p!^{1/q'}\}}(\RR^d)$, with $q'>q$, is continuously and densely injected into $\SSS^{\{M_p\}}_{\{p!^{1/q}\}}(\RR^d)$, $\dot{\mathcal{B}}^{\{M_p\}}(\RR^d)$ is the closure of $\SSS^{\{M_p\}}_{\{p!^{1/q'}\}}(\RR^d)$ in $\DD^{\{M_p\}}_{L^{\infty}}(\RR^d)$, for any $q'>0$. We mention several important properties of the space $\dot{\mathcal{B}}^{\{M_p\}}(\RR^d)$, $\DD'^{\{M_p\}}_{L^1}(\RR^d)$ and $\DD^{\{M_p\}}_{L^{\infty}}(\RR^d)$ and refer to \cite[Section 4.3]{dpv-2} for the complete account. First, $\dot{\mathcal{B}}^{\{M_p\}}(\RR^d)$ is a complete barrelled $(DF)$-space and thus $\DD'^{\{M_p\}}_{L^1}(\RR^d)$ is an $(F)$-space. The strong dual of the latter is topologically isomorphic to $\DD^{\{M_p\}}_{L^{\infty}}(\RR^d)$ (i.e. $\DD^{\{M_p\}}_{L^{\infty}}(\RR^d)$ is the strong bidual of $\dot{\mathcal{B}}^{\{M_p\}}(\RR^d)$). Denote by $\DD^{\{M_p\}}_{L^{\infty},c}(\RR^d)$ the space $\DD^{\{M_p\}}_{L^{\infty}}(\RR^d)$ equipped with the topology of compact convex circled convergence from the duality $\langle \DD'^{\{M_p\}}_{L^1}(\RR^d),\DD^{\{M_p\}}_{L^{\infty}}(\RR^d)\rangle$. Then $\DD^{\{M_p\}}_{L^{\infty},c}(\RR^d)$ is a complete space whose topology is weaker then the original one, $\dot{\mathcal{B}}^{\{M_p\}}(\RR^d)$ is continuously and densely injected into $\DD^{\{M_p\}}_{L^{\infty},c}(\RR^d)$ and the strong dual of $\DD^{\{M_p\}}_{L^{\infty},c}(\RR^d)$ is topologically isomorphic to $\DD'^{\{M_p\}}_{L^1}(\RR^d)$ (see \cite[Section 5]{BSV}).

\subsection{Convolution of Gelfand-Shilov ultradistributions}
\label{structure and topology}

Let $0<q\leq 1$ and $S,T\in\SSS'^{\{M_p\}}_{\{p!^{1/q}\}}(\RR^d)$. The $\SSS'^{\{M_p\}}_{\{p!^{1/q}\}}$-convolution of $S$ and $T$ is said to exists if $(S\otimes T)\varphi^{\Delta}\in \DD'^{\{M_p\}}_{L^1}(\RR^{2d})$, for all $\varphi\in\SSS^{\{M_p\}}_{\{p!^{1/q}\}}(\RR^d)$, where $\varphi^{\Delta}(x,y)=\varphi(x+y)$. When this is the case, $S*T\in \SSS'^{\{M_p\}}_{\{p!^{1/q}\}}(\RR^d)$ is defined as
\beqs
\langle S*T,\varphi\rangle ={}_{\DD'^{\{M_p\}}_{L^1}(\RR^{2d})}\langle (S\otimes T)\varphi^{\Delta},1\rangle_{\DD^{\{M_p\}}_{L^{\infty},c}(\RR^{2d})},\quad \varphi\in\SSS^{\{M_p\}}_{\{p!^{1/q}\}}(\RR^d),
\eeqs
where $1$ denotes the constant one function; see the comments after \cite[Definition 5.7]{BSV} for the reason why $S*T$ is indeed a continuous functional on $\SSS^{\{M_p\}}_{\{p!^{1/q}\}}(\RR^d)$. Furthermore, we have the following result.

\begin{theorem}\label{krclvn137}\cite[Theorem 5.8 and Remark 5.9]{BSV}
Let $0<q\leq 1$ and $S,T\in\SSS'^{\{M_p\}}_{\{p!^{1/q}\}}(\RR^d)$. The following statements are equivalent:
\begin{itemize}
\item[$(i)$] the $\SSS'^{\{M_p\}}_{\{p!^{1/q}\}}$-convolution of $S$ and $T$ exists;
\item[$(ii)$] for all $\varphi\in\SSS^{\{M_p\}}_{\{p!^{1/q}\}}(\RR^d)$, $(\varphi*\check{T})S\in\DD'^{\{M_p\}}_{L^1}(\RR^d)$;
\item[$(iii)$] for all $\varphi\in\SSS^{\{M_p\}}_{\{p!^{1/q}\}}(\RR^d)$, $(\varphi*\check{S})T\in\DD'^{\{M_p\}}_{L^1}(\RR^d)$.
\end{itemize}
Furthermore, when the $\SSS'^{\{M_p\}}_{\{p!^{1/q}\}}$-convolution of $S$ and $T$ exists,
\beq\label{motiv}
\langle S*T,\varphi\rangle ={}_{\DD'^{\{M_p\}}_{L^1}(\RR^d)}\langle (\varphi*\check{T})S,1\rangle_{\DD^{\{M_p\}}_{L^{\infty},c}(\RR^d)}= {}_{\DD'^{\{M_p\}}_{L^1}(\RR^d)}\langle (\varphi*\check{S})T,1\rangle_{\DD^{\{M_p\}}_{L^{\infty},c}(\RR^d)}.
\eeq
\end{theorem}

Our goal is to study the convolvability with $e^{s\langle \cdot\rangle^q}$, $q\geq 1$. We will adopt Theorem \ref{krclvn137} $(ii)$ as a definition for the existence of the  convolution with $e^{s\langle \cdot\rangle^q}$. More precisely:

\begin{definition}\label{defgen}
Let $q_1>q\geq 1$ and $s\in\RR\backslash\{0\}$. The ultradistribution
$S\in {\mathcal S}'^{\{M_p\}}_{\{p!^{1/q_1}\}}(\mathbb R^d)$ admits a $\SSS'^{\{M_p\}}_{\{p!^{1/q_1}\}}$-convolution with the function $e^{s \langle\cdot\rangle^q}$ if $(\varphi\ast e^{s \langle\cdot\rangle^q})S\in {\mathcal D}'^{\{M_p\}}_{L^1}(\mathbb R^d)$, for all
$\varphi\in {\mathcal S}^{\{M_p\}}_{\{p!^{1/q_1}\}}(\mathbb R^d)$. When this is the case, $S*e^{s \langle\cdot\rangle^q}\in\SSS'^{\{M_p\}}_{\{p!^{1/q_1}\}}(\RR^d)$ is defined by
\beq\label{def-con-wit-f}
\langle S\ast e^{s\langle \cdot\rangle^q},\varphi\rangle:=
{}_{\DD'^{\{M_p\}}_{L^1}(\RR^d)}\langle (\varphi*e^{s\langle \cdot\rangle^q})S,1\rangle_{\DD^{\{M_p\}}_{L^{\infty},c}(\RR^d)},\quad \varphi\in\SSS^{\{M_p\}}_{\{p!^{1/q_1}\}}(\RR^d).
\eeq
\end{definition}

To see that \eqref{def-con-wit-f} indeed defines a continuous functional on $\SSS^{\{M_p\}}_{\{p!^{1/q_1}\}}(\RR^d)$, it is enough to prove that the mapping
\beq\label{map-for-con-cc}
\varphi\mapsto (\varphi*e^{s\langle \cdot\rangle^q})S,\quad \SSS^{\{M_p\}}_{\{p!^{1/q_1}\}}(\RR^d)\rightarrow \DD'^{\{M_p\}}_{L^1}(\RR^d),
\eeq
is continuous (as $\DD'^{\{M_p\}}_{L^1}(\RR^d)$ is the strong dual of $\DD^{\{M_p\}}_{L^{\infty},c}(\RR^d)$; see the comments before this section). Clearly $\varphi\mapsto (\varphi*e^{s\langle x\rangle^q})S$, $\SSS^{\{M_p\}}_{\{p!^{1/q_1}\}}(\RR^d)\rightarrow \SSS'^{\{M_p\}}_{\{p!^{1/q_1}\}}(\RR^d)$, is continuous and thus has a closed graph. As $(\varphi\ast e^{s \langle\cdot\rangle^q})S\in {\mathcal D}'^{\{M_p\}}_{L^1}(\mathbb R^d)$, $\forall \varphi\in\SSS^{\{M_p\}}_{\{p!^{1/q_1}\}}(\RR^d)$, its graph is closed in $\SSS^{\{M_p\}}_{\{p!^{1/q_1}\}}(\RR^d)\times \DD'^{\{M_p\}}_{L^1}(\RR^d)$. Now the Ptak closed graph theorem \cite[Theorem 8.5, p. 166]{Sch} implies that \eqref{map-for-con-cc} is continuous ($\DD'^{\{M_p\}}_{L^1}(\RR^d)$ is a Ptak space since it is an $(F)$-space; see \cite[Section 4.8, p. 162]{Sch}).

\begin{remark}
Let $q_1>q\geq 1$ and $s>0$. If $S$ is a measurable function which satisfies $e^{k\langle \cdot\rangle^q}S\in L^{\infty}(\RR^d)$, $\forall k>0$, then the ordinary convolution of $e^{k\langle \cdot\rangle^q}$ and $S$ exists. It is straightforward to verify that the $\SSS'^{\{M_p\}}_{\{p!^{1/q_1}\}}$-convolution also exists and they are the same. Similarly, when $s<0$, if $S$ is a measurable function such that $e^{-\varepsilon\langle \cdot\rangle^q}S\in L^{\infty}(\RR^d)$, $\forall \varepsilon>0$, then both the ordinary and the $\SSS'^{\{M_p\}}_{\{p!^{1/q_1}\}}$-convolution of $S$ and $e^{s\langle \cdot\rangle^q}$ exist and they coincide.
\end{remark}

\section{Sufficient conditions for the convolution with $e^{s\langle \cdot\rangle^q}$, $q\geq 1$, $s\in\RR\backslash\{0\}$}\label{suff-cond-ee}

The aim of this section is to give a sufficient condition for an ultradistribution $S\in\SSS'^{\{M_p\}}_{\{p!^{1/q_1}\}}(\RR^d)$ to be $\SSS'^{\{M_p\}}_{\{p!^{1/q_1}\}}$-convolvable with $e^{s\langle \cdot\rangle^q}$, where $q_1>q\geq 1$ and $s\in\RR\backslash\{0\}$.\\
\indent Before we start, we recall the following multidimensional variant of the Fa\`a di Bruno formula.

\begin{proposition}(\cite[Corollary 2.10]{Faa})\label{faadibruno}
Let $|\alpha|=n\geq 1$ and $h(x_1,...,x_d)=f(g(x_1,...,x_d))$ with $g\in \mathcal{C}^{n}$ in a neighbourhood of $x^0$ and $f\in\mathcal{C}^n$ in a neighbourhood of $y^0=g(x^0)$. Then
\beqs
\partial^{\alpha}h(x^0)=\sum_{r=1}^{n}f^{(r)}(y^0)\sum_{p(\alpha,r)}\alpha!\prod_{j=1}^n \frac{\left(\partial^{\alpha^{(j)}}g(x^0)\right)^{k_j}}{k_j! \left(\alpha^{(j)}!\right)^{k_j}},
\eeqs
where
\beqs
p(\alpha,r)&=&\Big\{\left(k_1,...,k_n; \alpha^{(1)},...,\alpha^{(n)}\right)\Big|\, \mbox{for some}\, 1\leq s\leq n, k_j=0 \mbox{ and } \alpha^{(j)}=0\\
&{}&\mbox{for } 1\leq j\leq n-s;\, k_j>0 \mbox{ for } n-s+1\leq j\leq n; \mbox{ and}\\
&{}&0\prec \alpha^{(n-s+1)}\prec...\prec \alpha^{(n)} \mbox{ are such that}\\
&{}&\sum_{j=1}^n k_j=r,\, \sum_{j=1}^n k_j\alpha^{(j)}=\alpha\Big\}.
\eeqs
\end{proposition}

In the above formula, the convention $0^0=1$ is used. The relation $\prec$ used in this proposition is defined in the following way (cf. \cite{Faa}). We say that $\beta\prec\alpha$ when one of the following holds:
\begin{itemize}
\item[$(i)$] $|\beta|<|\alpha|$;
\item[$(ii)$] $|\beta|=|\alpha|$ and $\beta_1<\alpha_1$;
\item[$(iii)$] $|\beta|=|\alpha|$, $\beta_1=\alpha_1,...,\beta_k=\alpha_k$ and $\beta_{k+1}<\alpha_{k+1}$ for some $1\leq k<d$.
\end{itemize}

We need precise estimates for the growth of the derivatives of $e^{s\langle \cdot\rangle^q}$. For this purpose we first give a bound for the derivatives of $\langle \cdot\rangle^q$.

\begin{lemma}\label{skupa}
For every $q\in\RR$ there exists $C\geq 1$ such that
\begin{equation}\label{r-ta}
|\partial^{\alpha} \langle x\rangle^q|\leq   C^{|\alpha|}
 \alpha! \langle x\rangle^{q-|\alpha|},\quad \mbox{for all}\,\, x\in\RR^d,\, \alpha\in\NN^d.
\end{equation}
\end{lemma}

\begin{proof} We apply the Fa\'a di Bruno formula (Proposition \ref{faadibruno}) with $f(\rho)=\rho^q$ and $g(x)=\langle x\rangle$ to infer, for $\alpha\in\NN^d\backslash\{0\}$,
\beqs
|\partial^{\alpha}\langle x\rangle^q|\leq \sum_{m=1}^{|\alpha|} \langle x\rangle^{q-m}\left(\prod_{j=0}^{m-1} |q-j|\right) \alpha!\sum_{p(\alpha,m)}\prod_{j=1}^{|\alpha|}\frac{|\partial^{\alpha^{(j)}}\langle x\rangle|^{k_j}} {k_j!(\alpha^{(j)}!)^{k_j}}.
\eeqs
Notice that
\beqs
\prod_{j=0}^{m-1} |q-j|\leq\prod_{j=0}^{m-1} (|q|+j)\leq (\lfloor |q|\rfloor+m)!\leq 2^{|q|+m}\lfloor|q|\rfloor!m!.
\eeqs
Recall the bound (see \cite[$(0.1.1)$, p. 10]{NR})
\begin{equation}\label{prva1}
|\partial^{\beta}\langle x\rangle|\leq
2^{|\beta|+1}|\beta|!\langle x\rangle^{1-|\beta|}\leq 2(2d)^{|\beta|}\beta!\langle x\rangle^{1-|\beta|},\quad \mbox{for all}\,\, x\in\mathbb R^d,\, \beta\in\NN^d.
\end{equation}
Hence, we deduce
\beqs
|\partial^{\alpha}\langle x\rangle^q|\leq 2^{|q|}(2d)^{|\alpha|}\lfloor|q|\rfloor!\alpha!\langle x\rangle^{q-|\alpha|} \sum_{m=1}^{|\alpha|} 4^m m! \sum_{p(\alpha,m)}\prod_{j=1}^{|\alpha|}\frac{1} {k_j!}.
\eeqs
Now, \cite[Lemma 7.4]{PP3} gives
\beq\label{est-for-the-com-of-faadib}
\sum_{m=1}^{|\alpha|} m!\sum_{p(\alpha,m)}\prod_{j=1}^{|\alpha|}\frac{1}{k_j!} \leq 2^{|\alpha|(d+1)}.
\eeq
Plugging this in the above estimate, we deduce \eqref{r-ta}.
\end{proof}

\begin{lemma}\label{est-for-gaus-perviou}
Let $s\in\RR\backslash\{0\}$. Then for every $q>0$ there exists $C\geq1$ such that
\beqs
|\partial^{\alpha}e^{s\langle x\rangle^q}|\leq C^{|\alpha|} \alpha!e^{s\langle x\rangle^q}\sum_{m=1}^{|\alpha|} \frac{\langle x\rangle^{qm-|\alpha|}}{m!},\quad \mbox{for all}\,\, x\in\RR^d,\, \alpha\in\NN^d\backslash\{0\}.
\eeqs
\end{lemma}

\begin{proof} Let $\alpha\in\NN^d\backslash\{0\}$. We apply the Fa\'a di Bruno formula (Proposition \ref{faadibruno}) with $f(\rho)=e^{s\rho}$ and $g(x)=\langle x\rangle^q$ together with Lemma \ref{skupa} to infer ($C'_q\geq 1$ is the constant from this lemma)
\beqs
|\partial^{\alpha}e^{s\langle x\rangle^q}|&\leq& \alpha!\sum_{m=1}^{|\alpha|} |s|^m e^{s\langle x\rangle^q} \sum_{p(\alpha,m)}\prod_{j=1}^{|\alpha|}\frac{C'^{|\alpha^{(j)}|k_j}_q\langle x\rangle^{qk_j-|\alpha^{(j)}|k_j}} {k_j!}\\
&\leq& \alpha!((|s|+1)C'_q)^{|\alpha|}e^{s\langle x\rangle^q} \sum_{m=1}^{|\alpha|} \langle x \rangle^{qm-|\alpha|}\sum_{p(\alpha,m)}\prod_{j=1}^{|\alpha|}\frac{1} {k_j!}.
\eeqs
Now, the claim in the lemma follows from \eqref{est-for-the-com-of-faadib}.
\end{proof}

As an immediate consequence we have the following result.

\begin{corollary}\label{est-for-the-ggauss}
Let $s\in\RR\backslash\{0\}$. Then for every $q>0$ there exists $C\geq1$ such that
\beqs
|\partial^{\alpha}e^{s\langle x\rangle^q}|\leq C^{|\alpha|} \alpha!e^{s\langle x\rangle^q+ \langle x\rangle^{q-1}},\quad \mbox{for all}\,\, x\in\RR^d,\, \alpha\in\NN^d.
\eeqs
\end{corollary}

Before we prove the main result of this section, we need a technical lemma.

\begin{lemma}\label{lemma-for-the-est-of-exp}
For all $x,t\in\RR^d$, $q\geq 1$, $s\in\RR\backslash\{0\}$, the following inequality holds true:
\beq\label{est-for-the-con-if-exs}
s\langle x-t\rangle^q-s\langle x\rangle^q\leq q2^{q-1}|s||t|\langle x\rangle^{q-1}+q2^{q-1}|s||t|\langle t\rangle^{q-1}.
\eeq
\end{lemma}

\begin{proof} Let $\lambda_1,\lambda_2\geq0$. By the mean value theorem, there exists $\lambda'$ between $\lambda_1$ and $\lambda_2$ such that
\beq
s\langle \lambda_1\rangle^q-s\langle \lambda_2\rangle^q&=&sq(\lambda_1-\lambda_2)\lambda'\langle \lambda'\rangle^{q-2}\label{mean-val-the-for-the-gro} \\
&\leq& q|s||\lambda_1-\lambda_2|\langle \max\{\lambda_1,\lambda_2\}\rangle^{q-1}\label{est-for-the-nne}.
\eeq
We apply \eqref{est-for-the-nne} with $\lambda_1=|x-t|$ and $\lambda_2=|x|$. If $|x|\geq |x-t|$, we infer
\beqs
s\langle x-t\rangle^q-s\langle x\rangle^q\leq q|s||t|\langle x\rangle^{q-1}.
\eeqs
If $|x|\leq|x-t|$, then \eqref{est-for-the-nne} together with the inequality $\langle x-t\rangle^{q-1}\leq 2^{q-1}\langle x\rangle^{q-1}+2^{q-1}\langle t\rangle^{q-1}$ gives \eqref{est-for-the-con-if-exs}.
\end{proof}

\begin{theorem}\label{cond-for-exi-con-tt}
Let $q_1>q\geq 1$, $s\in\RR\backslash\{0\}$ and $S\in\SSS'^{\{M_p\}}_{\{p!^{1/{q_1}}\}}(\RR^d)$. If
\beq\label{cond-for-exist-of-conv}
e^{s\langle \cdot\rangle^q}e^{k\langle \cdot\rangle^{(q-1)q_1/(q_1-1)}} S\in \DD'^{\{M_p\}}_{L^1}(\RR^d),\quad \mbox{for all}\,\, k\geq 0,
\eeq
then the $\SSS'^{\{M_p\}}_{\{p!^{1/q_1}\}}$-convolution of $S$ and $e^{s\langle \cdot\rangle^q}$ exists.
\end{theorem}

\begin{remark}
Notice that $e^{k\langle \cdot\rangle^{(q-1)q_1/(q_1-1)}}$ is a multiplier for $\SSS'^{\{M_p\}}_{\{p!^{1/q_1}\}}(\RR^d)$ because $(q-1)q_1/(q_1-1)<q<q_1$ and Corollary \ref{est-for-the-ggauss}.
\end{remark}

\begin{remark}
When $q>1$, \eqref{cond-for-exist-of-conv} is equivalent to the following seemingly weaker condition
\beq\label{equl-for-ex-conv-on-ss}
e^{s\langle \cdot\rangle^q}e^{k\langle \cdot\rangle^{(q-1)q_1/(q_1-1)}} S\in \SSS'^{\{M_p\}}_{\{p!^{(q_1-1)/(q_1(q-1))}\}}(\RR^d),\quad \mbox{for all}\,\, k\geq 0.
\eeq
To see this, denote $q'=(q-1)q_1/(q_1-1)<q_1$ and assume that \eqref{equl-for-ex-conv-on-ss} holds. As
\beqs
e^{s\langle \cdot\rangle^q}e^{k\langle \cdot\rangle^{q'}} S =e^{-\langle \cdot\rangle^{q'}}(e^{s\langle \cdot\rangle^q}e^{(k+1)\langle \cdot\rangle^{q'}} S),
\eeqs
and $e^{-\langle \cdot\rangle^{q'}}\in \SSS^{\{M_p\}}_{\{p!^{1/q'}\}}(\RR^d)$ (because of Corollary \ref{est-for-the-ggauss}) it is enough to prove that for every $\psi\in \SSS^{\{M_p\}}_{\{p!^{1/q'}\}}(\RR^d)$ and $T\in\SSS'^{\{M_p\}}_{\{p!^{1/q'}\}}(\RR^d)$, $\psi T\in\DD'^{\{M_p\}}_{L^1}(\RR^d)$. Since $\varphi\mapsto \psi\varphi$, $\DD^{\{M_p\}}_{L^{\infty}}(\RR^d)\rightarrow \SSS^{\{M_p\}}_{\{p!^{1/q'}\}}(\RR^d)$, is continuous, the mapping $\varphi\mapsto \langle T, \psi\varphi\rangle$, $\dot{\mathcal{B}}^{\{M_p\}}(\RR^d)\rightarrow \CC$, is continuous, and $\psi T$ coincides with it on the dense subspace $\SSS^{\{M_p\}}_{\{p!^{1/q'}\}}(\RR^d)$ of $\dot{\mathcal{B}}^{\{M_p\}}(\RR^d)$. Consequently, $\psi T\in\DD'^{\{M_p\}}_{L^1}(\RR^d)$.
\end{remark}

\begin{proof} Let $\varphi\in\SSS^{\{M_p\}}_{\{p!^{1/q_1}\}}(\RR^d)$ be arbitrary but fixed. There exists $C,\eta>0$ such that $|\varphi(x)|\leq Ce^{-\eta\langle x\rangle^{q_1}}$, $\forall x\in\RR^d$. Corollary \ref{est-for-the-ggauss} implies
\beqs
|D^{\alpha}\varphi*e^{s\langle\cdot\rangle^q}(x)|\leq CC_1^{|\alpha|}\alpha!e^{2^q\langle x\rangle^{q-1}}\int_{\RR^d}e^{-\eta\langle t\rangle^{q_1}}e^{s\langle x-t\rangle^q}e^{2^q\langle t\rangle^{q-1}}dt.
\eeqs
Denote $\lambda=(\eta/(4^qq|s|))^{1/(q-1)}$ when $q>1$ and $\lambda=1$ when $q=1$. Lemma \ref{lemma-for-the-est-of-exp} gives
\beqs
|D^{\alpha}\varphi*e^{s\langle\cdot\rangle^q}(x)|\leq CC_1^{|\alpha|}\alpha!e^{2^q\langle x\rangle^{q-1}}e^{s\langle x\rangle^q}(I_1+I_2),
\eeqs
where
\beqs
I_1&=&\int_{|t|\leq (|x|/\lambda)^{(q-1)/(q_1-1)}}e^{-\eta\langle t\rangle^{q_1}}e^{q2^{q-1}|s||t|\langle x\rangle^{q-1}} e^{q2^{q-1}|s|\langle t\rangle^q} e^{2^q\langle t\rangle^{q-1}}dt,\\
I_2&=&\int_{|t|\geq (|x|/\lambda)^{(q-1)/(q_1-1)}}e^{-\eta\langle t\rangle^{q_1}}e^{q2^{q-1}|s||t|\langle x\rangle^{q-1}} e^{q2^{q-1}|s|\langle t\rangle^q}e^{2^q\langle t\rangle^{q-1}}dt.
\eeqs
Notice that
\beqs
I_1\leq e^{q2^{q-1}|s|\lambda^{-(q-1)/(q_1-1)}\langle x\rangle^{(q-1)q_1/(q_1-1)}}\|e^{-\eta\langle \cdot\rangle^{q_1}} e^{q2^{q-1}|s|\langle \cdot\rangle^q} e^{2^q\langle \cdot\rangle^{q-1}}\|_{L^1(\RR^d)}.
\eeqs
To estimate $I_2$, let $|t|\geq (|x|/\lambda)^{(q-1)/(q_1-1)}$. When $q>1$, we infer
\beqs
|t|\langle x\rangle^{q-1}\leq 2^q|t|+2^q|t||x|^{q-1}\leq 2^q|t|+2^q\lambda^{q-1}|t|^{q_1},
\eeqs
which implies
\beqs
I_2\leq \|e^{-(\eta/2)\langle \cdot\rangle^{q_1}} e^{q2^{2q-1}|s|\langle \cdot\rangle} e^{q2^{q-1}|s|\langle \cdot\rangle^q} e^{2^q\langle \cdot\rangle^{q-1}}\|_{L^1(\RR^d)}.
\eeqs
When $q=1$, one immediately deduces $I_2\leq e^2 \|e^{-\eta\langle \cdot\rangle^{q_1}} e^{2|s|\langle \cdot\rangle}\|_{L^1(\RR^d)}$. Consequently, there exists $c_1>0$ such that
\beqs
|D^{\alpha}\varphi*e^{s\langle\cdot\rangle^q}(x)|\leq C_2C_1^{|\alpha|}\alpha!e^{s\langle x\rangle^q} e^{c_1\langle x\rangle^{(q-1)q_1/(q_1-1)}},\quad \mbox{for all}\,\, x\in\RR^d,\, \alpha\in\NN^d.
\eeqs
Corollary \ref{est-for-the-ggauss} now verifies that $e^{-s\langle \cdot\rangle^q} e^{-(c_1+1)\langle \cdot\rangle^{(q-1)q_1/(q_1-1)}}(\varphi*e^{s\langle\cdot\rangle^q})\in \DD^{\{M_p\}}_{L^{\infty}}(\RR^d)$, which, in turn, gives
\beqs
(\varphi*e^{s\langle\cdot\rangle^q})S=\left(e^{-s\langle \cdot\rangle^q} e^{-(c_1+1)\langle \cdot\rangle^{(q-1)q_1/(q_1-1)}}(\varphi*e^{s\langle\cdot\rangle^q})\right) \left(e^{s\langle \cdot\rangle^q} e^{(c_1+1)\langle \cdot\rangle^{(q-1)q_1/(q_1-1)}}S\right)\in \DD'^{\{M_p\}}_{L^1}(\RR^d).
\eeqs
\end{proof}

\begin{remark}
Notice that
\beqs
\SSS^{\{M_p\}}_{\{p!^{1/q_1}\}}(\RR^d)=\lim_{\substack{\longrightarrow\\ h\rightarrow 0^+}} \tilde{\SSS}^{M_p,h}_{p!^{1/q_1},h}(\RR^d),
\eeqs
where $\tilde{\SSS}^{M_p,h}_{p!^{1/q_1},h}(\RR^d)$ is the Banach space of all $\varphi\in\mathcal{C}^{\infty}(\RR^d)$ such that
\beqs
\sup_{\alpha,\beta\in\NN^d} h^{|\alpha|+|\beta|}\|\langle \cdot\rangle^{|\beta|}\partial^{\alpha}\varphi\|_{L^{\infty}(\RR^d)} /(M_{\alpha}\beta!^{1/q_1})<\infty.
\eeqs
Thus, if $\{M_p\}_{p\in\NN}$ satisfies the non-quasianalyticity condition $\sum_{p=1}^{\infty}M_p^{-1/p}<\infty$, one can consider the case when $q_1=\infty$ as well; in fact $\SSS^{\{M_p\}}_{\{p!^0\}}(\RR^d)$ is exactly $\DD^{\{M_p\}}(\RR^d)$, i.e. the space of compactly supported ultradifferentiable functions of class $\{M_p\}$ (see \cite{Komatsu1}). In this case (and when $\{M_p\}_{p\in\NN}$ additionally satisfies the condition $(M.3)$ from \cite{Komatsu1}), when $q=2$, the condition \eqref{cond-for-exist-of-conv}, which boils down to $e^{s\langle \cdot\rangle^2}e^{k\langle \cdot\rangle}S\in\DD'^{\{M_p\}}_{L^1}(\RR^d)$, $\forall k\geq 0$, is known to be both necessary and sufficient for the existence of the $\DD'^{\{M_p\}}$-convolution with the Gaussian $e^{s\langle \cdot\rangle^2}$; see \cite[Theorem 4.3 $(a)$]{pp}. (See also \cite{Wagner} for the necessary and sufficient conditions for existence of convolution with the Gaussian $e^{s\langle \cdot\rangle^2}$ in the setting of Schwartz distributions.)
\end{remark}

\section{Necessary conditions for the convolution with $e^{s\langle \cdot\rangle^q}$, $q\geq 1$, $s>0$}\label{neces-con-sdd}

We now turn our attention to finding necessary conditions for the existence of the $\SSS'^{\{M_p\}}_{\{p!^{1/q_1}\}}$-convolution with $e^{s\langle \cdot\rangle^q}$. We start with two technical results.

\begin{lemma}\label{lemma-on-lower-boun}
Let $q_1>q\geq 1$, $s\in\RR\backslash\{0\}$ and $g$ be a measurable function on $\RR^d$ which satisfies the following: there exist $\eta_1,\eta_2,C'_1,C'_2>0$ such that
\beq\label{measur-fun-est}
C'_1e^{-\eta_1\langle x\rangle^{q_1}}\leq g(x)\leq C'_2 e^{-\eta_2\langle x\rangle^{q_1}},\quad \mbox{a.e.}
\eeq
Then for every $k>0$ there exists $c,\varepsilon\in(0,1)$ such that
\beqs
(g(\varepsilon\, \cdot)*e^{s\langle\cdot\rangle^q})(x)\geq c e^{s\langle x\rangle^q}e^{k\langle x\rangle^{(q-1)q_1/(q_1-1)}},\quad \forall x\in\RR^d.
\eeqs
\end{lemma}

\begin{proof} Notice first that when $q=1$ the claim immediately follows from
\beqs
\langle x\rangle -\langle t\rangle\leq \langle x-t\rangle\leq \langle x\rangle+\langle t\rangle,\quad \mbox{for all }\,\, x,t\in\RR^d.
\eeqs
So, throughout the rest of the proof we assume $q>1$; additionally, we denote by $B(x,r)$ the closed ball with centre at $x\in\RR^d$ and radius $r>0$. Assume first $s>0$ and let $k>0$ be arbitrary but fixed. Pick $k_1\geq 1+ 4(k+1)/(qs)$ and set $\varepsilon=1/(4^{q_1}(\eta_1+1)k_1)$. Let $x\in\RR^d$, $|x|\geq 2$, be arbitrary but fixed. For every $t\in B(-k_1|x|^{-1+(q-1)/(q_1-1)}x,1)$, we infer
\beqs
|x-t|&\geq& |x+k_1|x|^{-1+(q-1)/(q_1-1)}x|-|k_1|x|^{-1+(q-1)/(q_1-1)}x+t|\\
&\geq& |x|+k_1|x|^{(q-1)/(q_1-1)}-1\geq |x|+(k_1-1)|x|^{(q-1)/(q_1-1)}.
\eeqs
We apply the equality \eqref{mean-val-the-for-the-gro} with $\lambda_1=|x-t|$, $\lambda_2=|x|$. Because of the above, $\lambda_1>\lambda_2$ and thus $\lambda'\langle \lambda'\rangle^{q-2}\geq \langle x\rangle^{q-1}/2$. We infer
\beq\label{boun-for-the-epp}
s\langle x-t\rangle^q\geq s\langle x\rangle^q+(qs(k_1-1)/2)|x|^{(q-1)/(q_1-1)}\langle x\rangle^{q-1}\geq s\langle x\rangle^q+(k+1)\langle x\rangle^{(q-1)q_1/(q_1-1)},
\eeq
for all $t\in B(-k_1|x|^{-1+(q-1)/(q_1-1)}x,1)$. On the other hand, with $t$ as before, we have
\beqs
\eta_1\langle \varepsilon t\rangle^{q_1}\leq 2^{q_1}\eta_1\langle \varepsilon\rangle^{q_1}+2^{q_1}\eta_1\langle \varepsilon k_1|x|^{(q-1)/(q_1-1)}\rangle^{q_1}\leq 8^{q_1}\eta_1+\langle x\rangle^{(q-1)q_1/(q_1-1)}.
\eeqs
Thus, denoting by $\omega_d$ the measure of the unit ball in $\RR^d$, we deduce
\beq
(g(\varepsilon\, \cdot)*e^{s\langle\cdot\rangle^q})(x)&\geq& C'_1\int_{B(-k_1|x|^{-1+(q-1)/(q_1-1)}x,1)}e^{-\eta_1\langle \varepsilon t\rangle^{q_1}}e^{s\langle x-t\rangle^q}dt\nonumber \\
&\geq& \omega_d C'_1e^{-8^{q_1}\eta_1}e^{s\langle x\rangle^q}e^{k\langle x\rangle^{(q-1)q_1/(q_1-1)}}.\label{ine-for-co-fpt}
\eeq
Hence, the above holds for all $x\in\RR^d$, $|x|\geq 2$. As $g(\varepsilon\, \cdot)*e^{s\langle\cdot\rangle^q}$ is positive and continuous, this holds on all of $\RR^d$ as well possibly with a different constant in place of $\omega_d C'_1e^{-8^{q_1}\eta_1}$.\\
\indent Assume now $s<0$. Let $k>0$ be arbitrary but fixed. Pick $k_1\geq 1+4^q(k+1)/(q|s|)$ and set $\varepsilon=1/(4^{q_1}(\eta_1+1)k_1)$. There exists $c'\geq 4$ such that $k_1|x|^{-1+(q-1)/(q_1-1)}\leq 1/2$, for all $|x|\geq c'$. Let $x\in\RR^d$, $|x|\geq c'$, be arbitrary but fixed. For $t\in B(k_1|x|^{-1+(q-1)/(q_1-1)}x,1)$ we infer
\beqs
|x-t|\leq |x|(1-k_1|x|^{-1+(q-1)/(q_1-1)})+1=|x|-k_1|x|^{(q-1)/(q_1-1)}+1,
\eeqs
and thus $|x|\geq |x-t|+(k_1-1)|x|^{(q-1)/(q_1-1)}$. Furthermore
\beqs
|x-t|\geq |x|(1-k_1|x|^{-1+(q-1)/(q_1-1)})-1\geq |x|/2-1\geq |x|/4
\eeqs
(the very last inequality follows from $|x|\geq c'\geq 4$). We apply \eqref{mean-val-the-for-the-gro} with $\lambda_1=|x-t|$ and $\lambda_2=|x|$. Because of the above, $\lambda'\langle \lambda'\rangle^{q-2}\geq \langle x-t\rangle^{q-1}/2$ and consequently
\beqs
s\langle x-t\rangle^q -s\langle x\rangle^q&=&q|s|(|x|-|x-t|)\lambda'\langle \lambda'\rangle^{q-2}\geq q|s|((k_1-1)/2)|x|^{(q-1)/(q_1-1)}\langle x/4\rangle^{q-1}\\
&\geq& (k+1)\langle x\rangle^{(q-1)q_1/(q_1-1)},
\eeqs
for all $t\in B(k_1|x|^{-1+(q-1)/(q_1-1)}x,1)$. Now, similarly as above, one deduces the validity of \eqref{ine-for-co-fpt} for all $|x|\geq c'$. As $g(\varepsilon\, \cdot)*e^{s\langle\cdot\rangle^q}$ is positive and continuous this implies the claim in the lemma.
\end{proof}

\begin{lemma}\label{der-of-inv}
Let $q_1>q\geq 1$, $s>0$ and $g$ be a measurable function on $\RR^d$ which satisfies the following:
\begin{gather*}
\mbox{there exists}\,\,\, C',\eta'>0\,\,\, \mbox{such that}\,\,\, 0\leq g(x)\leq C' e^{-\eta'\langle x\rangle^{q_1}},\quad \mbox{a.e.}\\
\mathrm{measure}(\{x\in\RR^d|\, g(x)>0\})>0.
\end{gather*}
\begin{itemize}
\item[$(i)$] There exists $C>1$ such that
\beqs
|D^{\alpha}(g*e^{s\langle\cdot\rangle^q})(x)|\leq C^{|\alpha|+1}\alpha!\langle x\rangle^{(q-1)|\alpha|}(g*e^{s\langle\cdot\rangle^q})(x),\quad \mbox{for all}\,\, x\in\RR^d,\,\, \alpha\in\NN^d.
\eeqs
\item[$(ii)$] There exists $C>1$ such that
\beqs
|D^{\alpha}(1/(g*e^{s\langle\cdot\rangle^q})(x))|\leq C^{|\alpha|+1}\alpha!\langle x\rangle^{(q-1)|\alpha|}/(g*e^{s\langle\cdot\rangle^q})(x),\quad \mbox{for all}\,\, x\in\RR^d,\,\, \alpha\in\NN^d.
\eeqs
\end{itemize}
\end{lemma}

\begin{proof} The condition on $g$, as well as the fact that $s>0$, imply
\beqs
(g*e^{s\langle\cdot\rangle^q})(x)=\int_{\mathbb R^d} g(t) e^{s\langle x-t\rangle^{q}}dt\geq \|g\|_{L^1(\RR^d)}>0,\quad \forall x\in\RR^d.
\eeqs
To prove $(i)$ we apply Lemma \ref{est-for-gaus-perviou} with $\alpha\in\NN^d\backslash\{0\}$ to infer
\begin{gather*}
|D^{\alpha}(g*e^{s\langle\cdot\rangle^q})(x)|\leq C^{|\alpha|}\alpha!(I_1+I_2),\quad \mbox{where}\\
I_1=\int_{|t|\leq |x|}g(t) e^{s\langle x-t\rangle^q}\sum_{m=1}^{|\alpha|}\frac{\langle x-t\rangle^{qm-|\alpha|}}{m!}dt,\quad I_2=\int_{|t|\geq |x|}g(t) e^{s\langle x-t\rangle^q}\sum_{m=1}^{|\alpha|}\frac{\langle x-t\rangle^{qm-|\alpha|}}{m!}dt.
\end{gather*}
When $|t|\leq |x|$,
\beqs
\langle x-t\rangle^{qm-|\alpha|}\leq \langle x-t\rangle^{(q-1)|\alpha|}\leq 2^{(q-1)|\alpha|}\langle x\rangle^{(q-1)|\alpha|};
\eeqs
whence $I_1\leq e2^{(q-1)|\alpha|}\langle x\rangle^{(q-1)|\alpha|}(g*e^{\langle \cdot\rangle^q})(x)$. To estimate $I_2$, we proceed as follows
\beqs
I_2\leq \int_{|t|\geq |x|}g(t) e^{s\langle x-t\rangle^q}e^{\langle x-t\rangle^{q-1}}dt\leq \int_{|t|\geq |x|}g(t) e^{s(2\langle t\rangle)^q}e^{(2\langle t\rangle)^{q-1}}dt\leq \|g e^{s2^q\langle \cdot\rangle^q}e^{2^{q-1}\langle \cdot\rangle^{q-1}}\|_{L^1(\RR^d)}.
\eeqs
Now, $(i)$ follows from the fact that $g*e^{s\langle \cdot\rangle}$ is bounded from below by a positive constant.\\
\indent To prove $(ii)$ we apply the Fa\'a di Bruno formula (Proposition \ref{faadibruno}) to the composition of the functions $\rho\mapsto 1/\rho$ with $x\mapsto (g*e^{s\langle\cdot\rangle^q})(x)$. The first part of the lemma gives
\beqs
|D^{\alpha}(1/(g*e^{s\langle\cdot\rangle^q})(x))|\leq \alpha!\sum_{m=1}^{|\alpha|} \frac{C^{|\alpha|+m}m!\langle x\rangle^{(q-1)|\alpha|}}{(g*e^{s\langle\cdot\rangle^q})(x)}\sum_{p(\alpha,m)} \prod_{j=1}^{|\alpha|}\frac{1}{k_j!}.
\eeqs
Now, the bound in $(ii)$ follows from \eqref{est-for-the-com-of-faadib}.
\end{proof}

\begin{remark}
Under the assumptions of the above lemma, by applying Corollary \ref{est-for-the-ggauss}, it is straightforward to verify that
\beqs
|D^{\alpha}(g*e^{s\langle\cdot\rangle^q})(x)|\leq C^{|\alpha|+1}\alpha!e^{s'\langle x\rangle^q},\quad \mbox{for all},\,\, x\in\RR^d,\,\, \alpha\in\NN^d,
\eeqs
for some $C>1$, $s'>s$; i.e. $g*e^{s\langle\cdot\rangle^q}$ is a multiplier for $\SSS'^{\{M_p\}}_{\{p!^{1/q_1}\}}(\RR^d)$.
\end{remark}

Now we are ready to prove the main results of this section; recall that $p!^{\kappa}\subset M_p$ means: there exist $C,L>0$ such that $p!^{\kappa}\leq CL^pM_p$, $\forall p\in\NN$.

\begin{theorem}
Let $q_1>q\geq 1$, $s>0$ and $g$ a measurable function on $\RR^d$ which satisfies \eqref{measur-fun-est}.
\begin{itemize}
\item[$(i)$] Assume $p!^{2-1/q}\subset M_p$. If $S\in {\mathcal S}'^{\{M_p\}}_{\{p!^{1/q_1}\}}(\mathbb R^d)$ satisfies $(g(\varepsilon\,\cdot)*e^{s\langle \cdot\rangle^q})S\in\DD'^{\{M_p\}}_{L^1}(\RR^d)$, $\forall \varepsilon\in(0,1)$, then
    \beq\label{cond-for-ex-ss}
    e^{s'\langle \cdot\rangle^q}S\in \DD'^{\{M_p\}}_{L^1}(\RR^d),\quad \mbox{for all}\,\, s'<s.
    \eeq
    If $q=1$, then \eqref{cond-for-ex-ss} holds true even for $s'=s$.
\item[$(ii)$] Assume $p!^{2-1/q_1}\subset M_p$. If $S\in {\mathcal S}'^{\{M_p\}}_{\{p!^{1/q_1}\}}(\mathbb R^d)$ satisfies $(g(\varepsilon\,\cdot)*e^{s\langle \cdot\rangle^q})S\in\DD'^{\{M_p\}}_{L^1}(\RR^d)$, $\forall \varepsilon\in(0,1)$, then
    \beq\label{exis-conv-ex-ssuf}
    e^{s\langle \cdot\rangle^q}e^{k\langle \cdot\rangle^{(q-1)q_1/(q_1-1)}}S\in \DD'^{\{M_p\}}_{L^1}(\RR^d),\quad \mbox{for all}\,\, k\geq 0.
    \eeq
\end{itemize}
\end{theorem}

\begin{proof} First we prove $(ii)$. Let $k\geq 0$ be arbitrary but fixed. By Lemma \ref{lemma-on-lower-boun} there exists $c,\varepsilon\in(0,1)$ such that
\beq\label{bound-for-the-res}
(g(\varepsilon\,\cdot)*e^{s\langle\cdot\rangle^q})(x)\geq ce^{s\langle x\rangle^q}e^{(k+2)\langle x\rangle^{(q-1)q_1/(q_1-1)}},\quad \forall x\in\RR^d.
\eeq
Lemma \ref{der-of-inv} $(ii)$ implies
\beq\label{est-for-low-bbbb}
|D^{\alpha}(1/(g(\varepsilon\, \cdot)*e^{s\langle\cdot\rangle^q})(x))|\leq C_1C^{|\alpha|}\alpha!\langle x\rangle^{(q-1)|\alpha|}e^{-s\langle x\rangle^q}e^{-(k+2)\langle x\rangle^{(q-1)q_1/(q_1-1)}}.
\eeq
Since $t^N/N!\leq e^t$ for all $N\in\mathbb N$ and $t\geq 0$, we deduce
$$\langle x\rangle^{(q-1)|\alpha|}=|\alpha|!^{1-1/q_1}\left(\frac{\langle x\rangle^{|\alpha|(q-1)q_1/(q_1-1)}}{|\alpha|!}\right)^{1-1/q_1}\leq |\alpha|!^{1-1/q_1}e^{(1-1/q_1)\langle x\rangle^{(q-1)q_1/(q_1-1)}}$$
and then
\beqs
|D^{\alpha}(1/(g(\varepsilon\, \cdot)*e^{s\langle\cdot\rangle^q})(x))|\leq C_1C^{|\alpha|}|\alpha|!^{2-1/q_1}e^{-s\langle x\rangle^q}e^{-(k+1)\langle x\rangle^{(q-1)q_1/(q_1-1)}}.
\eeqs
Now, Corollary \ref{est-for-the-ggauss} yields $e^{s\langle \cdot\rangle^q}e^{k\langle \cdot\rangle^{(q-1)q_1/(q_1-1)}}/g(\varepsilon\, \cdot)*e^{s\langle\cdot\rangle^q}\in \DD^{\{p!^{2-1/q_1}\}}_{L^{\infty}}(\RR^d)$ and consequently
\beqs
e^{s\langle \cdot\rangle^q}e^{k\langle \cdot\rangle^{(q-1)q_1/(q_1-1)}}S=\frac{e^{s\langle \cdot\rangle^q}e^{k\langle \cdot\rangle^{(q-1)q_1/(q_1-1)}}}{g(\varepsilon\, \cdot)*e^{s\langle\cdot\rangle^q}} \left((g(\varepsilon\, \cdot)*e^{s\langle\cdot\rangle^q})S\right)\in\DD'^{\{M_p\}}_{L^1}(\RR^d).
\eeqs
To prove $(i)$, assume first $q>1$. Let $s'<s$ be arbitrary but fixed and denote $\kappa=s-s'>0$. Pick $\varepsilon>0$ such that \eqref{bound-for-the-res} holds true with $k=0$. By employing the bound
\beqs
\langle x\rangle^{(q-1)|\alpha|}\leq \kappa^{-|\alpha|(1-1/q)}|\alpha|!^{1-1/q}e^{(1-1/q)\kappa\langle x\rangle^q},
\eeqs
similarly as above we deduce
\beqs
|D^{\alpha}(1/(g(\varepsilon\, \cdot)*e^{s\langle\cdot\rangle^q})(x))|\leq C_2C^{|\alpha|}|\alpha|!^{2-1/q}e^{-s'\langle x\rangle^q}e^{-2\langle x\rangle^{(q-1)q_1/(q_1-1)}}.
\eeqs
Now, one can apply the same technique as before to conclude the validity of $(i)$. When $q=1$, notice that \eqref{est-for-low-bbbb} boils down to $|D^{\alpha}(1/(g(\varepsilon\, \cdot)*e^{s\langle\cdot\rangle})(x))|\leq C_1C^{|\alpha|}\alpha!e^{-s\langle x\rangle}$, which, in view of Corollary \ref{est-for-the-ggauss}, immediately implies $e^{s\langle \cdot\rangle}/g(\varepsilon\, \cdot)*e^{s\langle\cdot\rangle}\in \DD^{\{p!\}}_{L^{\infty}}(\RR^d)$; the rest of the proof is the same as before.
\end{proof}

Applying the above theorem for $g=e^{-\langle\cdot\rangle^{q_1}}$ we immediately deduce the following result.

\begin{corollary}\label{exis-con-cor-ls}
Let $q_1>q\geq 1$ and $s>0$.
\begin{itemize}
\item[$(i)$]  Assume $p!^{2-1/q}\subset M_p$. If the $\SSS'^{\{M_p\}}_{\{p!^{1/q_1}\}}$-convolution of $S\in {\mathcal S}'^{\{M_p\}}_{\{p!^{1/q_1}\}}(\mathbb R^d)$ and $e^{s\langle \cdot\rangle^q}$ exists then \eqref{cond-for-ex-ss} holds true. When $q=1$, \eqref{cond-for-ex-ss} holds true even for $s'=s$.
\item[$(ii)$] Assume $p!^{2-1/q_1}\subset M_p$. If the $\SSS'^{\{M_p\}}_{\{p!^{1/q_1}\}}$-convolution of $S\in {\mathcal S}'^{\{M_p\}}_{\{p!^{1/q_1}\}}(\mathbb R^d)$ and $e^{s\langle \cdot\rangle^q}$ exists then \eqref{exis-conv-ex-ssuf} holds true.
\end{itemize}
\end{corollary}

As a consequence of the above result together with Theorem \ref{cond-for-exi-con-tt} we have the following corollary.

\begin{corollary}\label{exi-con-whe-qq}
Let $q_1>1$, $s>0$ and $S\in\SSS'^{\{M_p\}}_{\{p!^{1/q_1}\}}(\RR^d)$. The $\SSS'^{\{M_p\}}_{\{p!^{1/q_1}\}}$-convolution of $S$ and $e^{s\langle \cdot\rangle}$ exists if and only if $e^{s\langle \cdot\rangle}S\in \DD'^{\{M_p\}}_{L^1}(\RR^d)$.
\end{corollary}

Since $(q-1)q_1/(q_1-1)<q$, Corollary \ref{exis-con-cor-ls} together with Theorem \ref{cond-for-exi-con-tt} also imply the following result.

\begin{corollary}\label{cort-l-stv}
Let $q_1>q\geq 1$, $s>0$, $p!^{2-1/q}\subset M_p$ and $S\in\SSS'^{\{M_p\}}_{\{p!^{1/q_1}\}}(\RR^d)$. If the $\SSS'^{\{M_p\}}_{\{p!^{1/q_1}\}}$-convolution of $S$ and $e^{s\langle \cdot\rangle^q}$ exists then the $\SSS'^{\{M_p\}}_{\{p!^{1/q_1}\}}$-convolution of $S$ and $e^{s'\langle \cdot\rangle^q}$ also exists for all $s'<s$, $s'\neq 0$.
\end{corollary}

\end{document}